\documentclass[leqno,12pt]{article} 
\usepackage{amsmath, amssymb}
\usepackage{amsthm} 
\theoremstyle{plain} 
\newtheorem{theorem}{\indent\sc Theorem}[section]
\newtheorem{lemma}[theorem]{\indent\sc Lemma}
\newtheorem{corollary}[theorem]{\indent\sc Corollary}
\newtheorem{proposition}[theorem]{\indent\sc Proposition}

\theoremstyle{definition} 

\newtheorem{remark}[theorem]{\indent\sc Remark}
\newtheorem{example}[theorem]{\indent\sc Example}

\newcommand\on{\operatorname}
\renewcommand\div{\on{div}}
\newcommand\grad{\on{grad}}
\newcommand\Hess{\on{Hess}}
\newcommand\Ric{\on{Ric}}
\newcommand\scal{\on{scal}}
\newcommand\vol{\on{vol(M)}}

\title{Gradient solitons on statistical manifolds}
\author{Adara M. Blaga and Bang-Yen Chen}

\date{}
\begin{document}

\maketitle

\markboth{{\small\it {\hspace{4cm} Gradient solitons on statistical manifolds}}}{\small\it{Gradient solitons on statistical manifolds
\hspace{4cm}}}


\begin{abstract}
We provide necessary and sufficient conditions for some par\-tic\-u\-lar couples $(g,\nabla)$ of pseudo-Riemannian metrics and affine con\-nec\-tions to be statistical structures if we have gradient almost Einstein, almost Ricci, almost Yamabe solitons, or a more general type of solitons on the manifold. In particular cases, we establish a formula for the volume of the manifold and give a lower and an upper bound for the norm of the Ricci curvature tensor field.
\end{abstract}

\bigskip

\section{Introduction}

Information geometry was firstly studied by Amari \cite{a85} treating the properties of the geometrical structures, such as Riemannian metrics and affine con\-nec\-tions, that naturally arise on a space of probability distributions. In this way, statistical structures
provide a link between information geometry and affine geometry. Such a Riemannian metric is the Fisher information metric defined on the manifold of probability distributions \cite{t}. According to Chentson's theorem, up to rescaling, the Fisher information metric on statistical models is the only Riemannian metric that is invariant under sufficient statistics \nolinebreak \cite{am}.

In differential geometry, a \textit{statistical structure} on a smooth manifold $M$ con\-sists of a pseudo-Riemannian metric $g$ and a torsion-free affine connection $\nabla$ such that $\nabla g$ is a Codazzi tensor field. For such a pair $(g,\nabla)$, the dual connection $\nabla^*$ of $\nabla$ with respect to $g$ is uniquely defined by $$X(g(Y,Z))=g(\nabla_XY,Z)+g(Y,\nabla^*_XZ),$$ for $X$, $Y$, $Z\in \mathfrak{X}(M)$. The pair $(\nabla,\nabla^*)$ is said to be a dualistic structure which plays an important role in statistics.

A particular statistical structure when the curvature tensor field of $\nabla$ vanishes is the \textit{Hessian structure} \cite{h}. In this case, also the curvature tensor field of the dual connection $\nabla^*$ vanishes and the manifold is called \textit{dually flat}.

\bigskip

We extend the notion of statistical structure in the following ways:

i) if $h$ is a symmetric $(0,2)$-tensor field and $\nabla$ is a torsion-free affine con\-nec\-tion, we call $(h,\nabla)$ a \textit{nearly statistical structure} on $M$ if $$({\nabla}_Xh)(Y,Z)=({\nabla}_Yh)(X,Z),$$ for any $X$, $Y$, $Z\in \mathfrak{X}(M)$;

ii) if $h$ is a $(0,2)$-tensor field and $\nabla$ is an affine connection, we call $(h,\nabla)$ a \textit{quasi-statistical structure} on $M$
if $d^{\nabla}h=0$  \cite{ma}, where $d^{\nabla}$ is defined by
$$(d^{\nabla}h)(X,Y,Z):=({\nabla}_Xh)(Y,Z)-({\nabla}_Yh)(X,Z)+h(T^{\nabla}(X,Y),Z),$$ for any $X$, $Y$, $Z\in \mathfrak{X}(M)$.

\bigskip

Regarded as stationary solutions of a geometric flow, the notion of \textit{solitons} can be generalized in some very natural ways. One question is that {\it if some of these generalizations are coming from certain particular flows, what geometrical and topological properties of the manifold shall reveal?}

\bigskip

In the context of statistical geometry, we will investigate in this article the consequences of the existence of different kind of solitons; such as Ricci, Einstein, Yamabe or a more general type defined by an affine connection, with a special view towards curvature. By means of an arbitrary $1$-form, we consider a statistical structure, which is equiaffine if the $1$-form is exact, and study some ge\-o\-met\-ri\-cal and topological properties of the solitons defined by it. Precisely, we provide a lower and an upper bound for the Ricci curvature tensor's norm, and in the compact case, using the classical Bochner formula, we determine a relation for the volume of the manifold. It is known that the Ricci tensor is the component of the curvature tensor of spacetime, related to the matter content of the universe via Einstein's field equation, its lower bounds allow us to deduce global geometrical properties of the manifold.

\section{Solitons and statistical structures}

Consider a pseudo-Riemannian manifold $(M,g)$ and let $\nabla^g$ be the Levi-Civita connection of $g$. We denote by $Q$ the Ricci operator defined by $g(QX,Y):=\Ric(X,Y)$, where $\Ric$ is the Ricci tensor of $(M,g)$.
If the Ricci tensor is non-degenerate, then it is a pseudo-Riemannian metric and we denote by $\nabla^{\Ric}$ the Levi-Civita connection of $\Ric$.

From Koszul's formula, we deduce:
$$2g(\nabla^{\Ric}_XY-\nabla^g_XY,QZ)=g((\nabla^g_YQ)X,Z)+g((\nabla^g_XQ)Y,Z)-g((\nabla^g_ZQ)X,Y),$$
for any $X$, $Y$, $Z\in \mathfrak{X}(M)$.

\medskip
Remark the following facts \cite{b}:

\vskip.06in
i) $(\Ric,\nabla^g)$ is a statistical structure if and only if $(\nabla^g_XQ)Y=(\nabla^g_YQ)X$, for any $X$, $Y\in \mathfrak{X}(M)$.
Moreover, if the Ricci operator $Q$ is a Codazzi tensor, then $QT=\frac{1}{2}\nabla^gQ$, where $T:=\nabla^{\Ric}-\nabla^g$;

\vskip.06in
ii) if $(g,\nabla^{\Ric})$ is a statistical structure, then $g(X,T(Y,Z))=g(Y,T(X,Z))$, for any $X$, $Y$, $Z\in \mathfrak{X}(M)$, where $T:=\nabla^{\Ric}-\nabla^g$;
\vskip.06in

iii) if $(g,\nabla)$ is a quasi-statistical structure, then $(\Ric,\nabla)$ is a quasi-statistical structure if and only if
$g((\nabla_XQ)Z,Y)=g((\nabla_YQ)Z,X)$, for any $X$, $Y$, $Z\in \mathfrak{X}(M)$.

\bigskip

Let $f$ be a smooth function on $M$. If the Hessian of $f$, denoted by $\Hess(f)$, is non-degenerate and of constant signature, then $\Hess(f)$ is a pseudo-Riemannian metric. A nice geometrical interpretation of Hessian metrics has recently appeared in mirror symmetry \cite{hi}, their practical importance being also shown in \cite{a}.

If we denote by $\nabla^{\Hess(f)}$ the Levi-Civita connection of $\Hess(f)$, then it follows from Koszul's formula that
$$
2g(\nabla^{\Hess(f)}_XY-\nabla^g_XY,\nabla_Z^g \nabla f)=g((\nabla^g)^2_{X,Y}\nabla f,Z)+g((\nabla^g)^2_{Y,Z}\nabla f,X)-g((\nabla^g)^2_{Z,X}\nabla f,Y),
$$
for any $X$, $Y$, $Z\in \mathfrak{X}(M)$, where $\nabla f$ denotes the gradient of $f$ and $(\nabla^g)^2_{X,Y}:=\nabla^g_X\nabla^g_Y-\nabla^g_{\nabla^g_XY}$, and we prove:

\begin{theorem}
$(\Hess(f),\nabla^g)$ is a statistical structure on $M$ if and only if the radial curvature vanishes, i.e. $R^{\nabla^g}(X,Y)\nabla f=0$, for any $X$, $Y\in \mathfrak{X}(M)$.
\end{theorem}
\begin{proof}
$d^{\nabla^g}\Hess(f)=0$ is equivalent to $$(\nabla_X^g \Hess(f))(Y,Z)=(\nabla_Y^g \Hess(f))(X,Z),$$ for any $X$, $Y$, $Z\in \mathfrak{X}(M)$, which gives
\pagebreak
$$X(g(\nabla_Y^g \nabla f, Z))-g(\nabla^g_{\nabla_X^gY}\nabla f, Z)-g(Y,\nabla^g_{\nabla_X^gZ}\nabla f)=$$
$$=Y(g(\nabla_X^g \nabla f, Z))-g(\nabla^g_{\nabla_Y^gX}\nabla f, Z)-g(X,\nabla^g_{\nabla_Y^gZ}\nabla f)$$
$$\iff X(g(\nabla_Y^g \nabla f, Z))-g(\nabla^g_{\nabla_X^gY}\nabla f, Z)-g(\nabla_X^gZ,\nabla^g_{Y}\nabla f)=$$
$$=Y(g(\nabla_X^g \nabla f, Z))-g(\nabla^g_{\nabla_Y^gX}\nabla f, Z)-g(\nabla_Y^gZ,\nabla^g_{X}\nabla f)$$
$$\iff g(\nabla_X^g\nabla_Y^g \nabla f, Z)-g(\nabla^g_{\nabla_X^gY}\nabla f, Z)=g(\nabla_Y^g\nabla_X^g \nabla f, Z)-g(\nabla^g_{\nabla_Y^gX}\nabla f, Z)$$
$$\iff g((\nabla^g)^2_{X,Y} \nabla f-(\nabla^g)^2_{Y,X} \nabla f, Z)=0$$
which is equivalent to
$R^{\nabla^g}(X,Y)\nabla f=0$.
\end{proof}

Note that the radial curvature was introduced by Klingenberg \cite{K} in the context of algebraic topology, to prove a homotopy sphere theorem.

\begin{theorem}
If $(g,\nabla)$ is a statistical structure, then $(\Hess(f),\nabla)$ is a statistical structure if and only if
$R^{\nabla}(X,Y)\nabla f=0,$
for any $X$, $Y\in \mathfrak{X}(M)$.
\end{theorem}
\begin{proof}
$d^{\nabla}\Hess(f)=0$ is equivalent to $$(\nabla_X \Hess(f))(Y,Z)-(\nabla_Y \Hess(f))(X,Z)+\Hess(f)(T^{\nabla}(X,Y),Z)=0,$$ for any $X$, $Y$, $Z\in \mathfrak{X}(M)$, which gives
\begin{align*}
X(g(\nabla_Y \nabla f, Z))-g(\nabla_{\nabla_XY}\nabla f, Z)-g(Y,\nabla_{\nabla_XZ}\nabla f)
-Y(g(\nabla_X \nabla f, Z))+\; & \\
+g(\nabla_{\nabla_YX}\nabla f, Z)+g(X,\nabla_{\nabla_YZ}\nabla f)+g(T^{\nabla}(X,Y),\nabla_Z \nabla f)=&\,0.
\end{align*}

Since $\nabla$ is torsion-free, we can express its curvature in terms of the second order derivatives, namely $R^{\nabla}(X,Y)=\nabla^2_{X,Y}-\nabla^2_{Y,X}$, where $\nabla^2_{X,Y}:=\nabla_X\nabla_Y-\nabla_{\nabla_XY}$ and the above equation becomes:
$$(\nabla_Xg)(\nabla_Y \nabla f, Z)-(\nabla_Yg)(\nabla_X \nabla f, Z)+g(R^{\nabla}(X,Y)\nabla f,Z)=0$$
which, from $d^{\nabla}g=0$, is equivalent to
$R^{\nabla}(X,Y)\nabla f=0$.
\end{proof}

\smallskip

Next, we shall relate statistical structures to gradient solitons (see also \cite{b}).
Recall that, for a pseudo-Riemannian metric $g$ and two smooth functions $f$ and $\lambda$, the triple $(g,f,\lambda)$ is called:

\vskip.06in
i) \textit{a gradient almost Ricci soliton} if
$$\Hess(f)+\Ric=\lambda g,$$
where $\Hess(f)$ is the Hessian of $f$ and $\Ric$ is the Ricci tensor of $g$;
\pagebreak

\vskip.06in
ii) \textit{a gradient almost Einstein soliton} if
$$\Hess(f)+\Ric=\left(\lambda+\frac{\scal}{2}\right) g,$$
where $\scal$ is the scalar curvature of $(M,g)$;

\vskip.06in
iii) \textit{a gradient almost Yamabe soliton} if
$$
\Hess(f)=(\lambda-\scal) g.
$$

In particular, if $\lambda$ is a constant, then we drop the adjective ``almost'' from the previous definitions and call the solitons the \textit{gradient Ricci}, \textit{gradient Einstein} and \textit{gradient Yamabe}, respectively.

Taking now the covariant derivative in the soliton equations, we obtain re\-spec\-tive\-ly:
$$(\nabla^g_X \Hess(f))(Y,Z)+(\nabla^g_X \Ric)(Y,Z)=X(\lambda)g(Y,Z),$$
$$(\nabla^g_X \Hess(f))(Y,Z)+(\nabla^g_X \Ric)(Y,Z)=X\left(\lambda+\frac{\scal}{2}\right)g(Y,Z),$$
$$(\nabla^g_X \Hess(f))(Y,Z)=X(\lambda-\scal)g(Y,Z),$$
for any $X$, $Y$, $Z\in \mathfrak{X}(M)$ and we can state:

\begin{proposition}
\begin{itemize}
\item[i)]
If $(g,f,\lambda)$ defines a gradient Ricci soliton, then
$(\Hess(f),\nabla^g)$ is a statistical structure on $M$ if and only if $(\Ric, \nabla^g)$ is a statistical structure on
\nolinebreak
$M$.

\item[ii)]If $(g,f,\lambda)$ defines a gradient Einstein soliton and $M$ is of constant scalar curvature, then $(\Hess(f),\nabla^g)$ is a statistical structure on $M$ if and only if $(\Ric, \nabla^g)$ is a statistical structure on $M$.

\item[iii)]If $(g,f,\lambda)$ defines a gradient Yamabe soliton and $M$ is of constant scalar curvature, then $(\Hess(f),\nabla^g)$ is a statistical structure on $M$.
\end{itemize}
\end{proposition}

We deduce the followings:

\begin{proposition}\label{p1}
If $(g,f,\lambda)$ defines a gradient almost Einstein soliton on the smooth manifold $M$ with non-degenerate Ricci tensor, then $(\Ric, \nabla^g)$ is a statistical structure on $M$ if and only if
\begin{equation}\label{e1}
g(R^{\nabla^g}(X,Y)\nabla f,Z)=X\left(\lambda+\frac{\scal}{2}\right)g(Y,Z)-Y\left(\lambda+\frac{\scal}{2}\right)g(X,Z),
\end{equation}
for any $X$, $Y$, $Z\in \mathfrak{X}(M)$.
\end{proposition}

\begin{proposition}
If $(g,f,\lambda)$ defines a gradient almost Ricci soliton on the smooth manifold $M$ with non-degenerate Ricci tensor, then $(\Ric, \nabla^g)$ is a statistical struc\-ture on $M$ if and only if
\begin{equation}\label{e2}
g(R^{\nabla^g}(X,Y)\nabla f,Z)=X(\lambda)g(Y,Z)-Y(\lambda)g(X,Z),
\end{equation}
for any $X$, $Y$, $Z\in \mathfrak{X}(M)$.
\end{proposition}

\begin{proposition}
If $(g,\nabla)$ is a statistical structure on the smooth manifold $M$ and $(g,f,\lambda)$ defines a gradient almost Einstein soliton on $M$ with non-degenerate Ricci tensor, then $(\Ric, \nabla)$ is a statistical structure on $M$ if and only if
\begin{equation}
g(\nabla^2_{X,Z}\nabla f, Y)-g(\nabla^2_{Y,Z}\nabla f, X)=X\left(\lambda+\frac{\scal}{2}\right)g(Y,Z)-Y\left(\lambda+\frac{\scal}{2}\right)g(X,Z),
\end{equation}
for any $X$, $Y$, $Z\in \mathfrak{X}(M)$, where $\nabla^2_{X,Y}Z:=\nabla_X\nabla_YZ-\nabla_{\nabla_XY}Z$.
\end{proposition}

\bigskip

From the soliton equations, we deduce respectively the followings:
$$\nabla^g \xi +Q=\lambda I,$$
$$\nabla^g \xi +Q=\left(\lambda+\frac{\scal}{2}\right) I,$$
$$\nabla^g \xi =(\lambda-\scal) I,$$
where $Q$ stands for the Ricci operator and $\xi:=\nabla f$. These lead to a more general notion of soliton, precisely
we consider
an almost $(\nabla,J,\xi, \lambda)$-soliton on $M$ as a data $(\nabla,J,\xi,\lambda)$ which satisfy the equation:
\begin{equation}\label{e7}
\nabla \xi+J=\lambda I,
\end{equation}
where $\nabla$ is an affine connection, $J$ is a $(1,1)$-tensor field, $\xi$ is a vector field and $\lambda$ is a smooth function on $M$.

A straightforward computation gives:

\begin{lemma}
If $(\nabla,J,\xi,\lambda)$ defines an almost $(\nabla,J,\xi)$-soliton on the pseudo-Rie\-man\-ni\-an manifold $(M,g)$, then the $2$-form $\Omega:=g(J \cdot,\cdot)$ is symmetric if and only if the endomorphism $\nabla\xi$ is self-adjoint with respect to $g$, i.e.
$$g(\nabla_X\xi,Y)=g(X, \nabla_Y\xi),$$
for any $X$, $Y\in \mathfrak{X}(M)$.
\end{lemma}

\begin{lemma}{\rm\cite{b}}\label{le}
The $2$-form $\Omega:=g(J \cdot,\cdot)$ satisfies $(\nabla_X\Omega)(Y,Z)=(\nabla_Y\Omega)(X,Z)$ if and only if
$$(\nabla_Xg)(JY,Z)-(\nabla_Yg)(JX,Z)=g((\nabla_YJ)X-(\nabla_XJ)Y,Z).$$
In particular, $\Omega$ is a Codazzi tensor field, i.e. $(\nabla^g_X\Omega)(Y,Z)=(\nabla^g_Y\Omega)(X,Z)$, if and only if $J$ is a Codazzi tensor field, i.e. $(\nabla^g_XJ)Y=(\nabla^g_YJ)X$.
\end{lemma}

\begin{remark}
If $\Omega$ is a Codazzi tensor field and $J$ is a Killing tensor field (i.e. $(\nabla^g_XJ)X=0$, for any $X\in \mathfrak{X}(M)$), then $J$ is $\nabla^g$-parallel.
\end{remark}

As particular cases, we deduce from \cite{b} the followings:

\begin{proposition}\label{p6}
Let $(\nabla,J,\xi,\lambda)$ define an almost $(\nabla,J,\xi)$-soliton on a pseudo-Riemannian manifold $(M,g)$. If $\Omega:=g(J \cdot,\cdot)$ is symmetric and $\nabla$ is torsion-free, then $(\Omega, \nabla)$ is a nearly statistical structure on $M$ if and only if
$$g(R^{\nabla}(X,Y)\xi,Z)=(\nabla_Xg)(JY,Z)-(\nabla_Yg)(JX,Z)+g(X(\lambda)Y-Y(\lambda)X, Z),$$
for any $X$, $Y$, $Z\in \mathfrak{X}(M)$.
\end{proposition}

\begin{corollary}
If $(\nabla^g,J,\xi,\lambda)$ defines an almost $(\nabla^g,J,\xi)$-soliton on a pseudo-Riemannian manifold $(M,g)$ and $\Omega:=g(J \cdot,\cdot)$ is symmetric, then $(\Omega, \nabla^g)$ is a nearly statistical structure on $M$ if and only if
$$R^{\nabla^g}(\cdot,\cdot)\xi=d\lambda\otimes I-I\otimes d\lambda.$$
\end{corollary}

Assume now $\xi=\nabla f$ and from the soliton equation (\ref{e7}) we get:
$$\Hess^{\nabla}(f)+\Omega=\lambda g,$$
hence, $\nabla^g_X \Hess^{\nabla}(f)+\nabla^g_X \Omega=X(\lambda)g$, for any $X\in \mathfrak{X}(M)$.

\begin{corollary}
Let $(\nabla^g,J,\xi,\lambda)$ define an almost $(\nabla^g,J,\xi)$-soliton on the pseudo-Riemannian manifold $(M,g)$ with $\lambda$ a constant and $\xi=\nabla f$. Then the following statements are equivalent:

\begin{itemize}
\item[i)]$(\Omega, \nabla^g)$ is a nearly statistical structure on $M$;

\item[ii)]$R^{\nabla^g}(X,Y)\nabla f=0$, for any $X$, $Y\in \mathfrak{X}(M)$;

\item[iii)]$(\Hess^{\nabla}(f), \nabla^g)$ is a nearly statistical structure on $M$.
\end{itemize}
\end{corollary}

\section{Connections defined by $1$-forms and solitons}

Inspired by the property of projectively equivalence of connections, given an arbitrary $1$-form $\eta$ on the pseudo-Riemannian manifold $(M,g)$, we consider
the affine connection:
$$\nabla^{\eta}:=\nabla^g+\eta\otimes I+I\otimes {\eta}+g\otimes \xi,$$
where $\nabla^g$ is the Levi-Civita connection of $g$ and $\xi$ is the $g$-dual vector field of $\eta$ (i.e. $\eta=i_{\xi}g$).
We get:
$$T^{\nabla^{\eta}}=0, \ \ d^{\nabla^{\eta}}g=0.$$
Hence, we have

\begin{proposition}
For any $1$-form $\eta$ on the pseudo-Riemannian manifold $(M,g)$, the pair $(g,\nabla^{\eta})$ is a statistical structure on $M$ and $\nabla^{-\eta}$ is the dual connection of $\nabla^{\eta}$.
\end{proposition}

In particular, $\nabla^{\eta}_{\xi}\xi=\nabla^g_{\xi}\xi+3| \xi|_g^2 \xi$, therefore, $\xi$ is a geodesic vector field for
$\nabla^{\eta}$ if and only if $\nabla^g_{\xi}\xi=-3| \xi|_g^2 \xi$.
Moreover:
\vskip.06in

i) $\xi$ is $\nabla^{\eta}$-parallel if and only if $(\nabla^g,J:=2\eta\otimes \xi,\xi,\lambda:=|\xi|_g^2)$ is a soliton;
\vskip.06in

ii) $\eta$ is $\nabla^{\eta}$-parallel if and only if $(\nabla^g,J:=-2\eta\otimes \xi,\xi,\lambda:=-|\xi|_g^2)$ is a soliton.

The curvature of the connection $\nabla^{\eta}$ is given by:
\begin{equation}\begin{aligned}\notag &\hskip.5in (R^{\nabla^{\eta}}-R^{\nabla^g})(X, Y)Z=[g(Y, \nabla ^g_X\xi)-g(X, \nabla ^g_Y\xi )]Z+
\\&\; +[\eta (Y)\eta (Z)+g(Y,Z)|\xi |_g^2-g(Z, \nabla ^g_Y\xi )]X+g(Y, Z)\nabla ^g_X\xi -g(X, Z)\nabla ^g_Y\xi-\\&-[\eta (X)\eta (Z)+g(X,Z)|\xi |_g^2-g(Z, \nabla ^g_X\xi )]Y+
[\eta (X)g(Y, Z)-\eta (Y)g(X, Z)]\xi,
\end{aligned}\end{equation}
and we deduce that, if $\xi$ is a $g$-null and $\nabla^g$-parallel vector field (i.e. $\eta(\xi)=0$ and $\nabla^g\xi=0$), condition appearing in Walker manifolds \cite{w:b}, then
$$(R^{\nabla^{\eta}}-R^{\nabla^g})(X, Y)Z\in \ker \eta,$$ for any $X$, $Y$, $Z\in \mathfrak{X}(M)$.

\begin{example}
Let $(M,\varphi,\xi,\eta,g)$ be a Kenmotsu manifold and let $\nabla^{\eta}$ be the affine connection defined by the structure,
$\nabla^{\eta}:=\nabla^g+\eta\otimes I+I\otimes {\eta}+g\otimes \xi$. Since $\nabla^g\xi=I-\eta\otimes \xi$, we get $\nabla^{\eta}\xi=2I+\eta\otimes \xi$, and $(\nabla^{\eta},J:=-\eta\otimes \xi, \lambda=2)$ is a soliton on $M$.
\end{example}

The divergence operator with respect to $\nabla^{df}$ is given by:
$$\div^{(g,\nabla^{df})}=\div^{(g,\nabla^g)}+(n+2)df.$$
In the compact case, it follows from the divergence theorem that
$$\int_M \div^{(g,\nabla^{df})}(X)d\mu_g=(n+2)\int_M g(\grad_g(f),X)d\mu_g,$$
for any $X\in \mathfrak{X}(M)$. Moreover, if $|\grad_g(f)|_g$ is constant, then:
$$\vol=\frac{1}{(n+2)|\grad_g(f)|_g^2}\int_M \div^{(g,\nabla^{df})}(\grad_g(f))d\mu_g.$$

Denote by $\Delta^g:=\div^{(g,\nabla^g)}\circ \grad_g$ and $\Delta^{\eta}:=\div^{(g,\nabla^{\eta})}\circ \grad_g$ the corresponding Laplace operators.
Then:
$$\Delta^{df}(\tilde{f})=\Delta^{g}(\tilde{f})+(n+2)g(\grad_g(f),\grad_g(\tilde{f})),$$
for any smooth function $\tilde{f}$ on $M$.

Note that if $\tilde{f}$ is harmonic for $\Delta^{df}$, then $\tilde{f}$ is also harmonic for $\Delta^{g}$ if and only if the vector fields $\grad_g(f)$ and $\grad_g(\tilde{f})$ are $g$-orthogonal.

In particular, we have
$$\Delta^{df}(f)=\Delta^{g}(f)+(n+2)| \grad_g(f)|_g^2,$$
hence:

\vskip.06in
i) if $f$ is harmonic for $\Delta^{g}$, then it is a subharmonic function for $\Delta^{df}$ (i.e. $\Delta^{df}(f)\geq 0$) provided $| \grad_g(f)|_g^2\geq 0$;

\vskip.06in
ii) if $f$ is harmonic for $\Delta^{df}$ and $(M,g)$ is a compact Riemannian manifold, then $f$ is locally constant.

\bigskip

Also, for any smooth function $f\in C^{\infty}(M)$, if we denote by $\Hess^{g}(f)$ and $\Hess^{\eta}(f)$ the Hessian tensor fields with respect to $\nabla^g$ and $\nabla^{\eta}$, then we have:
\begin{align*}
\Hess^{\eta}(f)(X,Y)&:=g(\nabla^{\eta}_X\grad_g(f),Y)=\\
&=\Hess^{g}(f)(X,Y)+
\eta(\grad_g(f))g(X,Y)+\eta(X) df(Y)+\eta(Y)df(X)
\end{align*}
and by tracing this relation we find
$$\Delta^{\eta}(f)=\Delta^{g}(f)+(n+2)\eta(\grad_g(f)),$$
where $n=\dim(M)$.
In particular, if $\xi=\grad_g(f)$, then $\eta=df$ and we obtain:
\begin{equation}\label{e5}
\Hess^{df}(f)=\Hess^{g}(f)+|\grad_g(f)|_g^2g+2df\otimes df.
\end{equation}

\bigskip

Recall that a pseudo-Riemannian manifold $(M,g)$ with a pair of dual con\-nec\-tions $(\nabla, \nabla^*)$ is called \textit{conjugate Ricci-symmetric} \cite{Min} if $\Ric^{\nabla}=\Ric^{\nabla^*}$.

\bigskip

If we denote by $\Ric^g$ and $\Ric^{\eta}$ the Ricci tensors for $\nabla^g$ and $\nabla^{\eta}$, then from the curvature relation we obtain that the Ricci curvature of $\nabla^{\eta}$ satisfies
\begin{align*}&\Ric^{\eta}(Y,Z)=\Ric^{g}(Y,Z)+g(Y,Z)\{n| \xi|_g^2+\div^{(g,\nabla^{g})}(\xi)\}+\\&\; \;\;+(n-2)\eta(Y)\eta(Z) +g(Y,\nabla_Z^g\xi)-(n+1)g(Z,\nabla_Y^g\xi),\end{align*}
where $n=\dim(M)$ and we can state:

\begin{proposition}\label{p111}
$(M, g,{\nabla}^{\eta},{\nabla}^{-\eta})$ is a conjugate Ricci-symmetric manifold.
\end{proposition}

An affine connection on $M$ is called \textit{equiaffine} \cite{ns} if it admits a parallel volume form on $M$.
It is known that \cite{ns} the necessary and sufficient condition for a torsion-free affine connection to be equiaffine is that the Ricci tensor is symmetric.

Since $\nabla^{\eta}$ is torsion-free, we get:
\begin{proposition}
$\nabla^{\eta}$ is an equiaffine connection on $M$ if and only if the en\-do\-mor\-phism $\nabla^g\xi$ is self-adjoint with respect to $g$, i.e.
$$g(\nabla_X^g\xi,Y)=g(X,\nabla_Y^g\xi),$$
for any $X$, $Y\in \mathfrak{X}(M)$.
\end{proposition}

In particular, if $\xi=\grad_g(f)$, then $\eta=df$ and we obtain:
\begin{equation}\label{e3}
\Ric^{df}=\Ric^{g}+\{n| \grad_g(f)|_g^2+\Delta^g(f)\}g+(n-2)df\otimes df-n\Hess^g(f).
\end{equation}
Hence we have

\begin{corollary}
$\nabla^{df}$ is an equiaffine connection on $M$.
\end{corollary}

Taking the trace in the previous relation, we get:
$$\scal^{(g,\nabla^{df})}=\scal^{(g,\nabla^g)}+ (n-1)(n+2)|\grad_g(f)|_g^2,$$
which implies $\scal^{(g,\nabla^{df})}\geq \scal^{(g,\nabla^g)}$ provided $|\grad_g(f)|_g^2\geq 0$.

\bigskip

If we denote by $Q^g$ and $Q^{df}$ the Ricci operators defined by $g(Q^{g}X,Y):=\Ric^{g}(X,Y)$ and $g(Q^{df}X,Y):=\Ric^{df}(X,Y)$, $X$, $Y\in \mathfrak{X}(M)$, then:
$$Q^{df}=Q^g+\{n| \grad_g(f)|_g^2+\Delta^g(f)\}I+(n-2)df\otimes \grad_g(f)-n\nabla^g\grad_g(f)$$
and by direct computations, we obtain:

\begin{proposition}
Let $(M,g)$ be an $n$-dimensional pseudo-Riemannian man\-i\-fold,
\linebreak
$\xi=\grad_g(f)$ and $\eta=df$. Then:

\begin{itemize}
\item[i)]$(\nabla^{\eta},Q^g, \xi,\lambda)$ is a gradient almost soliton if and only if $(\nabla^{g},Q^g+2\eta\otimes \xi, \lambda-|\xi|^2_g)$ is a gradient almost soliton;

\item[ii)]$(\nabla^g,Q^{\eta}, \xi,\lambda)$ is a gradient almost soliton if and only if $(\nabla^{g},\frac{1}{1-n}\{Q^g+(n-2)\eta\otimes \xi\},
\frac{1}{1-n}\{\lambda-n|\xi|^2_g-\Delta^g(f)\})$ is a gradient almost soliton;

\item[iii)]$(\nabla^{\eta},Q^{\eta}, \xi,\lambda)$ is a gradient almost soliton if and only if $(\nabla^{g},\frac{1}{1-n}(Q^g+n\eta\otimes \xi),
\frac{1}{1-n}\{\lambda-(n+1)|\xi|^2_g-\Delta^g(f)\})$ is a gradient almost soliton.
\end{itemize}
\end{proposition}


Now, we shall relate the previously considered types of solitons to almost Ricci and almost $\eta$-Ricci solitons \cite{blaga}.

\begin{proposition}
Let $(M,g)$ be an $n$-dimensional pseudo-Riemannian man\-i\-fold,
\linebreak
$\xi=\grad_g(f)$ and $\eta=df$. Then we have:

\begin{itemize}
\item[i)]$(\nabla^{\eta},Q^g, \xi,\lambda)$ is a gradient almost soliton if and only if \linebreak[2]$(g,\xi,\lambda-\nolinebreak|\xi|^2_g,-2)$ is a gradient almost $\eta$-Ricci soliton.

\item[ii)] If $\nabla^g\xi=\eta\otimes \xi$, then:

\begin{itemize}
\item[(ii.1)] $(\nabla^g,Q^{\eta}, \xi,\lambda)$ is a gradient almost soliton of $M$ if and only if $(g,\xi,\lambda-\nolinebreak (n+1)|\xi|^2_g,2)$ is a gradient almost $\eta$-Ricci soliton; in this case, $\scal^{(g,\nabla^{\eta})}=n\lambda -|\xi|^2_g$;

\item[(ii.2)]  $(\nabla^{\eta},Q^{\eta}, \xi,\lambda)$ is a gradient almost soliton of $M$ if and only if $(g,\xi,\lambda-(n+2)|\xi|^2_g)$ is a gradient almost Ricci soliton; in this case, $\scal^{(g,\nabla^g)}=n\lambda -(n+1)^2|\xi|^2_g$.
\end{itemize}
\end{itemize}\end{proposition}
\begin{proof}
i) $\nabla^{\eta}\xi+Q^{g}=\lambda I$ is equivalent to $\nabla^g\xi+Q^g=(\lambda-|\xi|^2_g)I-2\eta\otimes \xi$.

ii) From hypotheses we get $\Delta^g(f)=|\xi|^2_g$.

For (ii.1), by taking the trace in
$$-\Hess^{g}(f)+\Ric^{g}=\{\lambda-(n+1)|\xi|^2_g\}g,$$
we obtain $\scal^{(g,\nabla^g)}=n\lambda -(n^2+n-1)|\xi|^2_g$, therefore, $\scal^{(g,\nabla^{\eta})}=n\lambda -|\xi|^2_g$. By a similar proof we get the conclusion (ii.2).
\end{proof}

\smallskip

We shall further derive a formula for the volume of $M$ whenever it admits an almost soliton.

By computing the scalar product with respect to $g$, we find:
$$\langle \Ric^{df}, df\otimes df\rangle_g=\Ric^{g}(\grad_g(f),\grad_g(f))-n\Hess^{g}(f)(\grad_g(f),\grad_g(f))
+$$$$+|\grad_g(f)|^2_g\Delta^g(f)+2(n-1)|\grad_g(f)|^4_g$$
and using the classical Bochner formula, we obtain:
$$\langle \Ric^{df}, df\otimes df\rangle_g=\frac{1}{2}\Delta^g(|\grad_g(f)|^2_g)-|\Hess^{g}(f)|^2_g-\grad_g(f)(\Delta^g(f))-$$
$$-n \Hess^{g}(f)(\grad_g(f),\grad_g(f))+|\grad_g(f)|^2_g\Delta^g(f)+2(n-1)|\grad_g(f)|^4_g$$
and we can state:

\begin{proposition}\label{p}
Let $(M,g)$ be a compact $n$-dimensional pseudo-Riemannian man\-i\-fold, $f$ a smooth function on $M$ such that $|\grad_g(f)|_g$ is constant.
Then:
\begin{align*}
\vol&=\frac{1}{2(n-1)|\grad_g(f)|^4_g}\Bigg\{\int_M  |\Hess^{g}(f)|^2_g d\mu_g+\\
&\quad +\int_M \grad_g(f)(\Delta^g(f)) d\mu_g+ \int_M \langle \Ric^{df}, df\otimes df\rangle_gd\mu_g\Bigg\}.
\end{align*}
\end{proposition}

Also, the Bochner formula can be written in terms of $\Ric^{df}$ and $\Hess^{df}(f)$ as follows:
$$\frac{1}{2}\Delta^g(|\grad_g(f)|^2_g)=|\Hess^{df}(f)|^2_g+\Ric^{df}(\grad_g(f),\grad_g(f))+\grad_g(f)(\Delta^g(f))-$$
$$-3(n+2)|\grad_g(f)|^4_g-3|\grad_g(f)|^2_g \Delta^g(f)+\frac{n-4}{2}\grad_g(f)(|\grad_g(f)|^2_g)=$$
$$=|\Hess^{df}(f)|^2_g+\Ric^{df}(\grad_g(f),\grad_g(f))+\grad_g(f)(\Delta^{df}(f))-$$
$$-3|\grad_g(f)|^2_g \Delta^{df}(f)-\frac{n+8}{2}\grad_g(f)(|\grad_g(f)|^2_g)$$
and we can state:

\begin{proposition}
Let $(M,g)$ be a compact $n$-dimensional pseudo-Riemannian man\-i\-fold, $f$ a smooth function on $M$ such that $|\grad_g(f)|_g$ is constant.
Then:
\begin{align*}
\vol &=\frac{1}{3(n+2)|\grad_g(f)|^4_g}\Bigg\{\int_M  |\Hess^{df}(f)|^2_g d\mu_g+\\
&\quad +\int_M \grad_g(f)(\Delta^{df}(f)) d\mu_g+ \int_M \Ric^{df}(\grad_g(f),\grad_g(f))d\mu_g\Bigg\}.
\end{align*}
\end{proposition}

\begin{remark}
Notice that, under the same hypotheses, we have:
$$\int_M  |\Hess^{g}(f)|^2_g d\mu_g+
\int_M \grad_g(f)(\Delta^{g}(f)) d\mu_g+
\int_M \Ric^{g}(\grad_g(f),\grad_g(f))d\mu_g=0.
$$
\end{remark}

\begin{proposition}
Let $(M,g)$ be a compact $n$-dimensional pseudo-Riemannian manifold, $f$ a smooth function on $M$ such that $|\grad_g(f)|_g$ is constant.
If
$(\nabla^g,Q^{df}, \grad_g(f),\lambda)$ is a gradient almost soliton, then:
$$\vol=\frac{1}{2(n-1)|\grad_g(f)|^4_g}\Bigg\{\int_M  |\Hess^{g}(f)|^2_g d\mu_g+
 |\grad_g(f)|^2_g \int_M \lambda d\mu_g+$$$$+ n\int_M \grad_g(f)(\lambda)d\mu_g-\int_M \grad_g(f)(\scal^{(g,\nabla^g)})d\mu_g\Bigg\}.$$
\end{proposition}
\begin{proof}
Indeed, $\nabla^g \grad_g(f)+Q^{df}=\lambda I$ implies
\begin{equation}\label{e4}
\Hess^{g}(f)+\Ric^{df}=\lambda g.
\end{equation} Then:
$$\langle \Ric^{df}, df\otimes df\rangle_g=\lambda |\grad_g(f)|^2_g-\frac{1}{2}\grad_g(f)(|\grad_g(f)|^2_g).$$
Also, replacing $\Ric^{df}$ from (\ref{e3}) in (\ref{e4}) and taking the trace with respect to $g$, we get:
$$\Delta^g(f)+\scal^{(g,\nabla^g)}+(n-1)(n+2)|\grad_g(f)|^2_g=n\lambda.$$
Now, by applying $\grad_g(f)$ to the previous relation and using Proposition \ref{p}, we obtain the conclusion.
\end{proof}

\begin{remark}
Under the same hypotheses, we have:

\vskip.06in
i) if $\lambda=2(n-1)|\grad_g(f)|^2_g$ and $n\geq 3$, then
$$\vol=\frac{1}{(n-1)(n-2)|\grad_g(f)|^2_g}\int_M \scal^{(g,\nabla^g)}d\mu_g.$$
Moreover, if $\scal^{(g,\nabla^g)}$ is constant, then
$$\scal^{(g,\nabla^g)}=(n-1)(n-2)|\grad_g(f)|^2_g\geq 0$$
provided $|\grad_g(f)|^2_g\geq 0$ and $f$ is a harmonic function for $\Delta^g$;

\vskip.06in
ii) if $\lambda$ is a constant and $\lambda \neq 2(n-1)|\grad_g(f)|^2_g$, then
$$\vol=\frac{1}{\{2(n-1)|\grad_g(f)|^2_g-\lambda\}|\grad_g(f)|^2_g}\Bigg\{\int_M |\Hess^{g}(f)|^2_g d\mu_g-$$
$$-\int_M \grad_g(f)(\scal^{(g,\nabla^g)})d\mu_g\Bigg\}.$$
Moreover, if $f$ is a harmonic function for $\Delta^g$, then
$$\vol=\frac{1}{\{2(n-1)|\grad_g(f)|^2_g-\lambda\}|\grad_g(f)|^2_g}\int_M |\Hess^{g}(f)|^2_g d\mu_g.$$
Hence, in the Riemannian case, $\lambda < 2(n-1)|\grad_g(f)|^2_g$, therefore $$\scal^{(g,\nabla^g)}<(n-1)(n-2)|\grad_g(f)|^2_g.$$
\end{remark}

\begin{proposition}
Let $(M,g)$ be an $n$-dimensional pseudo-Riemannian man\-i\-fold and $f$ a smooth function on $M$.
If $(\nabla^g,Q^{df}, \grad_g(f),\lambda)$ is a gradient almost soliton, then:
$$(n-1)^2 |\Hess^{g}(f)|^2_g + \frac{(n-1)(n-2)^2}{n}|\grad_g(f)|^4_g-\frac{(n-1)^2}{n}(\Delta^g(f))^2+$$$$+\frac{2(n-1)(n-2)}{n}|\grad_g(f)|^2_g\Delta^g(f)-2(n-1)(n-2)\Hess^g(f)(\grad_g(f),\grad_g(f))\leq
$$
$$\leq |\Ric^{g}|^2_g \leq $$
$$\leq
(n-1)^2 |\Hess^{g}(f)|^2_g - \frac{(n-1)(n-2)^2}{n}|\grad_g(f)|^4_g+\frac{1}{n}(\scal^{(g,\nabla^g)})^2+$$
$$+\frac{2(n-2)}{n}|\grad_g(f)|^2_g\scal^{(g,\nabla^g)}-
2(n-2)\Ric^g(\grad_g(f),\grad_g(f))
.$$
\end{proposition}
\begin{proof} This proposition follows by
computing $|\Hess^{g}(f)|^2_g$ from
$$(1-n)\Hess^{g}(f)=\{\lambda-n|\grad_g(f)|^2_g-\Delta^g(f)\}g-(n-2)df\otimes df-\Ric^{g}$$
and $|\Ric^{g}|^2_g$ from
$$\Ric^{g}=\{\lambda-n|\grad_g(f)|^2_g-\Delta^g(f)\}g-(n-2)df\otimes df+(n-1)\Hess^{g}(f)$$
as well as the fact that the conditions to exist a solution (in $\lambda$) give precisely the double inequality from the conclusion.
\end{proof}

\begin{proposition}
Let $(M,g)$ be a compact $n$-dimensional pseudo-Riemann\-ian manifold, $f$ a smooth function on $M$ such that $|\grad_g(f)|_g$ is constant.
If
$(\nabla^{df},Q^{g}, \grad_g(f),\lambda)$ is a gradient soliton, then:
\begin{align*}& \vol=\frac{1}{(\lambda-3|\grad_g(f)|^2_g)|\grad_g(f)|^2_g}\Bigg\{-\int_M  |\Hess^{g}(f)|^2_g d\mu_g+
\\ &\hskip.3in  +\int_M \grad_g(f)(\scal^{(g,\nabla^g)})d\mu_g\Bigg\},
\end{align*}
provided $\lambda\neq 3|\grad_g(f)|^2_g$.
\end{proposition}
\begin{proof}
Indeed, $\nabla^{df} \grad_g(f)+Q^{g}=\lambda I$ implies
\begin{equation}\label{e6}
\Hess^{df}(f)+\Ric^{g}=\lambda g.
\end{equation} Then:
\begin{align*}& \langle \Ric^{df}, df\otimes df\rangle_g=\Ric^{g}(\grad_g(f),\grad_g(f))+
\\ &\hskip.3in  +\{n|\grad_g(f)|^2_g+\Delta^g(f)\}|\grad_g(f)|^2_g+(n-2)|\grad_g(f)|^4_g.
\end{align*}
Also, replacing $\Hess^{df}(f)$ from (\ref{e5}) in (\ref{e6}) and taking the trace with respect to $g$, we get:
$$\Delta^g(f)+\scal^{(g,\nabla^g)}+(n+2)|\grad_g(f)|^2_g=n\lambda.$$
Now, by applying $\grad_g(f)$ to the previous relation and using Proposition \ref{p}, we obtain the conclusion.
\end{proof}

\begin{proposition}
Let $(M,g)$ be an $n$-dimensional pseudo-Riemannian man\-i\-fold and $f$ a smooth function on $M$.
If $(\nabla^{df},Q^{g}, \grad_g(f),\lambda)$ is a gradient almost soliton, then:
$$|\Hess^{g}(f)|^2_g + \frac{4(n-1)}{n}|\grad_g(f)|^4_g-\frac{1}{n}(\Delta^g(f))^2-\frac{4}{n}|\grad_g(f)|^2_g\Delta^g(f)+$$$$+4\Hess^g(f)(\grad_g(f),\grad_g(f))\leq$$$$ \leq |\Ric^{g}|^2_g \leq $$$$\leq
|\Hess^{g}(f)|^2_g - \frac{4(n-1)}{n}|\grad_g(f)|^4_g+\frac{1}{n}(\scal^{(g,\nabla^g)})^2+\frac{4}{n}|\grad_g(f)|^2_g\scal^{(g,\nabla^g)}-$$$$-4\Ric^g(\grad_g(f),\grad_g(f))
.$$
\end{proposition}
\begin{proof} This proposition follows by
computing $|\Hess^{g}(f)|^2_g$ from
$$\Hess^{g}(f)=(\lambda-|\grad_g(f)|^2_g)g-2df\otimes df-\Ric^{g}$$
and $|\Ric^{g}|^2_g$ from
$$\Ric^{g}=(\lambda-|\grad_g(f)|^2_g)g-2df\otimes df-\Hess^{g}(f)$$
as well as the fact that the conditions to exist a solution (in $\lambda$) give precisely the double inequality from the conclusion.
\end{proof}

\begin{remark}
If $|\grad_g(f)|_g=1$, then the double equality from the previous Proposition implies
$$\Ric^g(\grad_g(f),\grad_g(f))=\frac{1}{4n}\{(\scal^{(g,\nabla^g)}+2)^2+(\Delta^g(f)+2)^2-8n\}.$$
In this case, from the soliton equation (\ref{e6}), we get
$$\Ric^g(\grad_g(f),\grad_g(f))=\lambda-3$$
and
$$\Delta^g(f)=n\lambda-(n+2)-\scal^{(g,\nabla^g)}.$$
Replacing the last two expressions in the first one and asking for the equation of order two in $\lambda$ to have solution, we get $n^2(\scal^{(g,\nabla^g)})^2\leq 0$, hence the manifold is of zero scalar curvature. Moreover, if $f$ is a harmonic function for $\Delta^g$, then $\lambda=\frac{n+2}{n}$ and $\Ric^g(\grad_g(f),\grad_g(f))=-\frac{2(n-1)}{n}<0$.

Note that Petersen {\rm\cite{pe}} called $f$ a \textit{distance function} if it is a solution of the Hamilton-Jacobi equation $|\grad_g(f)|^2_g=1$, which he has used in his book.
\end{remark}


\bigskip

\textit{Adara M. Blaga}

\textit{Department of Mathematics}

\textit{West University of Timi\c{s}oara}

\textit{Timi\c{s}oara, Rom\^{a}nia}

\textit{adarablaga@yahoo.com}


\bigskip

\textit{Bang-Yen Chen}

\textit{Department of Mathematics}

\textit{Michigan State University}

\textit{East Lansing, MI, USA}

\textit{chenb@msu.edu}

\end{document}